\documentclass{amsart}
\usepackage{amssymb,amsmath}
\usepackage{amsfonts}
\usepackage{hyperref}
\usepackage{subcaption}
\usepackage{graphicx}
\usepackage{array,booktabs}
\usepackage{multirow}
\usepackage{xparse}
\usepackage{enumerate}
\usepackage{bbm}
\usepackage{mathtools}
\usepackage{stmaryrd}
\usepackage[square,numbers]{natbib}
\ExplSyntaxOn
\DeclareExpandableDocumentCommand{\IfNoValueOrEmptyTF}{mmm}
 {
  \IfNoValueTF{#1}{#2}
   {
    \tl_if_empty:nTF {#1} {#2} {#3}
   }
 }
\ExplSyntaxOff

\newcommand\cev[1]{\overleftarrow{#1}}
\newcommand{\alphabet}{\mathcal{A}}
\DeclareDocumentCommand\pv{s o}{\IfBooleanTF{#1}{{\cev{\mathsf{p}}}}{\mathsf{p}}\IfNoValueF{#2}{^{(#2)}}}
\DeclareDocumentCommand\tv{s o}{\IfBooleanTF{#1}{{\cev{\mathsf{q}}}}{\mathsf{q}}\IfNoValueF{#2}{^{(#2)}}}
\DeclareDocumentCommand\tree{o}{\mathcal{T}\IfNoValueF{#1}{^{(#1)}}}
\DeclareDocumentCommand\tshift{o o}{T\IfNoValueOrEmptyTF{#1}{}{^{(#1)}}_{\IfNoValueF{#2}{#2}}}
\DeclareDocumentCommand\aprod{o o}{X\IfNoValueOrEmptyTF{#1}{}{^{\times #1}}_{\IfNoValueF{#2}{#2}}}
\DeclareDocumentCommand\level{o m o}{\Xi\IfNoValueF{#1}{^{(#1)}}_{#2\IfNoValueF{#3}{:#3}}}
\DeclareDocumentCommand\lattice{o m o}{\Delta\IfNoValueF{#1}{^{(#1)}}_{#2\IfNoValueF{#3}{:#3}}}
\DeclareDocumentCommand\block{m o}{B_{#1\IfNoValueF{#2}{:#2}}}
\DeclareDocumentCommand\pblock{m o}{Z_{#1\IfNoValueF{#2}{:#2}}}
\DeclareDocumentCommand\dv{m o}{\tau_{#1\IfNoValueF{#2}{:#2}}}
\DeclareDocumentCommand\dvset{m o}{D_{#1\IfNoValueF{#2}{:#2}}}
\newcommand{\norm}[1]{|#1|}
\DeclareDocumentCommand\pm{s o}{\IfBooleanTF{#1}{{\cev{\mathsf{P}}}}{\mathsf{P}}\IfNoValueF{#2}{^{(#2)}}}
\DeclareDocumentCommand\tm{s o}{\IfBooleanTF{#1}{{\cev{\mathsf{Q}}}}{\mathsf{Q}}\IfNoValueF{#2}{^{(#2)}}}
\DeclareDocumentCommand\qm{s o}{\IfBooleanTF{#1}{{\cev{\mathsf{R}}}}{\mathsf{R}}\IfNoValueF{#2}{^{(#2)}}}
\DeclareDocumentCommand\trm{m o}{\eta_{#1\IfNoValueF{#2}{:#2}}}
\DeclareDocumentCommand\trmset{m o}{S_{#1\IfNoValueF{#2}{:#2}}}
\newcommand{\childset}[1]{\sigma(#1)}
\DeclareDocumentCommand\proddom{s m o}{\IfBooleanTF{#1}{W}{\Omega}_{#2\IfNoValueF{#3}{:#3}}}
\newcommand{\onemat}{\mathbbm{1}}
\newcommand{\onevec}{\mathbbm{1}}
\newcommand{\stdvec}[1]{\mathsf{e}_{#1}}
\newcommand{\DKL}{D_{\mathrm{KL}}}

\newtheorem{theorem}{Theorem}

\newtheorem{corollary}[theorem]{Corollary}

\newtheorem{definition}[theorem]{Definition}

\newtheorem{lemma}[theorem]{Lemma}

\newtheorem{proposition}[theorem]{Proposition}
\newtheorem{remark}[theorem]{Remark}

\title[Topological pressure of axial product on trees]{On the topological pressure of axial product on trees}
\author{Jung-Chao Ban}
\author{Yu-Liang Wu}
\date{March 2022}

\begin{document}

\maketitle
\begin{abstract}
    This article investigates the topological pressure of isotropic axial products of Markov subshifts on the $d$-tree. We show that the quantity increases with dimension $d$. To achieve this, we introduce the pattern distribution vectors and the associated transition matrices and partially transplant the large deviation theory to tree-shifts. Additionally, we apply our main result to a broader class of shift spaces, accompanied by numerical experiments for verification.
\end{abstract}

\section{Introduction} \label{sec:introduction}
This paper investigates the topological pressure of the axial product of a Markov subshift on the $d$-tree, inspired by recent studies on the limiting entropy of the axial product of $\mathbb{N}^d$ \cite{louidor2013independence,meyerovitch2014independence} and the asymptotic pressure of the axial product on $d$-tree \cite{Petersen2021}. Before presenting the main results, the motivation behind the study is outlined below.

Let $\alphabet$ be a finite set with $\left\vert \alphabet\right\vert =k$ and let $X_{1},\ldots ,X_{d}\subseteq \alphabet^{\mathbb{N}}$ be one-sided subshifts. The associated \emph{axial product of subshifts $X_{1},\ldots ,X_{d}$ on $\mathbb{N}^{d}$}, denoted by $\otimes_{i=1}^{d}X_{i}=X_{1} \otimes \cdots \otimes X_{d} \subset \alphabet^{\mathbb{N}^{d}}$, is defined as
\begin{equation}
    \otimes_{i=1}^{d} X_{i}=\{x \in \alphabet^{\mathbb{N}^{d}}:\forall g \in \mathbb{N}^{d} \text{ }\forall i\in \{1,\ldots ,d\}, x_{g+\mathbb{Z}_+ \stdvec{i}}\in X_{i}\},  \label{3}
\end{equation}
where $x_{g + \mathbb{Z}_+ \stdvec{i}}\in \alphabet^{\mathbb{N}}$ is the sequence obtained by shifting $x$ by $g$, and $\{\stdvec{1},\ldots ,\stdvec{d}\}$ denotes the standard basis of $\mathbb{N}^{d}$. Let $\tree[d]$ be the rooted $d$-tree, which is the Cayley graph of a free monoid generated by $\{f_{1},\ldots,f_{d}\}$ with its identity element $\epsilon$ representing the \emph{root} of the tree. The \emph{axial product of subshifts $X_{1},\ldots ,X_{d} \subseteq \alphabet^{\mathbb{N}}$ on $\tree[d]$}, denoted by $\times_{i=1}^{d}X_{i}=X_{1}\times \cdots \times X_{d}$, is similarly defined as 
\begin{equation}
\times_{i=1}^{d}X_{i}=\{x\in \alphabet^{\tree[d]}:\forall g\in 
\tree[d]\text{ }\forall i\in \{1, \ldots, d\}, (x_{g f_{i}^n})_{n \in \mathbb{Z}_+}\in X_{i}\}\text{.}  \label{4}
\end{equation}
An axial product $\otimes_{i=1}^{d}X_{i}$ (or $\times _{i=1}^{d}X_{i}$) is said to be \emph{isotropic} if $X_{i}=X_{j}$ for all $1\leq i\neq j\leq d$, and \emph{anisotropic} if otherwise. Isotropic axial products
of shifts on $\mathbb{N}^{d}$ were first introduced in \cite{louidor2013independence}, and many important physical systems, such as the hard square model on $\mathbb{N}^{2}$ or $\mathbb{Z}^{2}$, belong to this class. In the context of tree-shifts, a notable subclass of isotropic axial products is the family of \emph{hom tree-shifts}\footnote{Such shifts are also called the associated tree shifts in \cite{Petersen2018b,Petersen2020}.}, where $X_i$ are one-sided Markov subshifts and identical. This family is widely studied in the literature \cite{aubrun2012tree,ban2017mixing,ban2017tree,benjamini1994markov,Petersen2018b,Petersen2020,Petersen2021} and has attracted growing attention since the subshifts defined on it exhibit rich and diverse phenomena in topological (cf.~\cite{ban2017mixing}) and statistical contexts (cf.~\cite{ban2021structure,ban2022topological}).

One of the key quantities that characterize such symbolic dynamical systems is topological entropy, which, in this context, measures the asymptotic growth in the number of admissible patterns with finite supports in the following manner: for a subshift $X \subseteq \alphabet^{G}$ ($G = \mathbb{N}^d$ or $\mathbb{Z}^d$), its \emph{topological entropy} is defined as 
\begin{equation}
h(X)=\lim_{n \to \infty }\frac{\log \norm{\pi([1,n]^d,X)}}{%
\left\vert [1,n]^d\right\vert }\text{,}  \label{1}
\end{equation}
where $\pi: G \times \alphabet^{G} \to \alphabet^{F}$ is the \emph{canonical projection} onto $F$ defined by $(\pi(F,x))_g = x_g$, $\left\vert F\right\vert$ indicates the number of elements in a set $F$, and the limit is known to exist because $\mathbb{N}^d$/$\mathbb{Z}^d$ is an amenable semigroup/group (cf.~\cite{ceccherini2010cellular}) and $\{[1,n]^d\}_{n=1}^{\infty}$ is a \emph{F\o{}lner sequence}. Analogously, we define the \emph{topological entropy} of a tree-shift $\tshift[d] \subseteq \alphabet^{\tree[d]}$ to be
\begin{equation}
h(\tshift[d])=\lim_{n \to \infty }\frac{\log \norm{\pi(\lattice[d]{n},\tshift[d])}}{\left\vert \lattice[d]{n}\right\vert }\text{,}  \label{2}
\end{equation}
where $\level[d]{n}=\{f_1,f_2,\cdots,f_d\}^n$, $\lattice[d]{n}=\cup_{i=0}^{n} \level[d]{i}$, and the existence of the limit (\ref{2}) was proved in \cite{Petersen2018b}. For more on the existence of the limit for shifts defined on a large class of trees, see \cite{Bana}. It is important to note that despite similar definitions of topological entropy for these two types of systems, the resulting quantities differ significantly in various ways. For instance, there are notable differences (see~\cite{ban2021structure}) in the structures of $\{h(\tshift):\tshift \text{ is a Markov tree-shift on } \tree[d]\}$ and $\{h(X):X \text{ is a Markov shift on } \mathbb{N}^{d}\}$.

The theme of this paper centers on the behavior of the topological entropy as the dimension of the underlying lattice grows. This concept was introduced by Louidor, Marcus and Pavlov \cite{louidor2013independence}, leading to the study of \emph{limiting entropy} defined as 
\[
h^{(\infty )}(X) = \lim_{d \to \infty} h(\otimes_{i=1}^d X),
\]
where the limit exists since $h(\otimes_{i=1}^d X)$ is non-increasing in $d$. Meyerovitch and Pavlov \cite{meyerovitch2014independence} later proved that this asymptotic quantity coincides with the independence entropy (see the references for a definition). A similar investigation on the limiting entropy of $\times_{i=1}^d X$ was initiated by Petersen and Salama. In \cite{Petersen2020}, they focused on $X_{G}$, the golden-mean\footnote{A shift space with two symbols $0, 1$ that preclude the existence of neighboring 1's.} subshift and employed the method of \emph{site strip approximation} to prove that $h(X_{G}^{\times d})$ is strictly increasing in $d$ (\cite[Theorem 3.7]{Petersen2020}). 

Recently, Petersen and Salama \cite{Petersen2021} extended the notion of limiting entropy to \emph{topological pressure\footnote{They refer to it as \emph{asymptotic pressure}.}}. Under a similar setting to \cite{burton1995variational}, they introduced the idea of topological pressure to isotropic axial products on a broad class of trees called generalized Fibonacci trees\footnote{See \cite{Petersen2021} for a definition.}. In this paper, we focus on the cases where the underlying lattices are $d$-trees. Under the circumstances, one may consider a physical system of particles sitting on the $d$-tree so that the energy assigned to a particle of type $a \in \alphabet$ comes solely from the influence of the ambient field $\log w_a$ ($w \in \mathbb{R}_{\ge 0}^{\alphabet}$) and the pairwise interaction $\log A_{a,b}$ ($A \in \mathbb{R}_{\ge 0}^{\alphabet \times \alphabet}$) with neighboring particle of type $b \in \alphabet$. Collectively, such information is embedded in a single matrix $E \in \mathbb{R}_{\ge 0}^{\alphabet \times \alphabet}$ defined as $E_{a,b} = w_a A_{a,b}$. Under these assumptions, the configurations adhering to the pairwise interaction constraint $A$ naturally form an isotropic axial product $\aprod[d][E]$, where
\[
\aprod[][E] = \{x \in \mathbb{Z}_+: E_{x_{n+1}, x_n} > 0, \forall n \in \mathbb{Z}_+\},
\]
and one can take the same route as in the study of thermodynamic formalism to define the \emph{partition function} $\norm{\pblock{n}(\aprod[d][E],E)}$ on the finite subsystem $\lattice{n}$ as
\begin{equation} \label{eq:partition_function}
    \norm{\pblock{n}(\aprod[d][E],E)} = \sum_{u \in \pi(\lattice[d]{n},\aprod[d][E])} w_{u_\epsilon} \prod_{g \in \lattice[d]{n-1}} \prod_{i=1}^{d} E_{u_{g f_i}, u_{g}}
\end{equation}
and the \emph{topological pressure} as
\begin{equation} \label{eq:topological_pressure}
    \mathbf{P}(\aprod[d][E],E) = \lim_{n \to \infty} \frac{\log \norm{\pblock{n}(\aprod[d][E], E)}}{\norm{\lattice[d]{n}}},
\end{equation}
where the existence of the limit was proved in \cite[Theorem 3.10]{Petersen2021}, and $\mathbf{P}(\aprod[d][E],E) = h(\aprod[d][E])$ if $E$ is a $0$-$1$ matrix. Moreover, the authors showed under the assumption
\begin{equation} \tag{A} \label{eq:assumption}
    \sum_{a \in \alphabet} E_{a,b} > 0 \quad \text{ and } \quad \sum_{b \in \alphabet} E_{a,b} > 0,
\end{equation}
that the topological pressure $\mathbf{P}(\aprod[d][E],E)$ converges asymptotically, as $d \to \infty$, to
\begin{equation}
\log r_{E} := \log \max \left\{\sum_{a \in \alphabet} E_{a,b}: b \in \alphabet\right\},  \label{5}
\end{equation}
generalizing the limiting entropy to limiting pressure. Building upon this, this paper further establishes that $\mathbf{P}(\aprod[E][d],E)$ is increasing in $d$ with the aid of \emph{pattern distribution}. Specifically, by letting $\Gamma_{\alphabet}$ represent the set of all probability vectors indexed by $\alphabet$ and $\Upsilon_{\alphabet}$ be the set of left stochastic matrices acting on $\Gamma_{\alphabet}$, we successfully relate the topological pressure to the maximum $P^{(\infty)}(d,E)$ of the following optimization problem:
\begin{equation} \label{eq:infinite_optimization}
    \begin{cases}
        \text{maximize} & \sum\limits_{j=0}^{\infty} \frac{d-1}{d^{j+1}} \sum\limits_{a \in \alphabet} \sum\limits_{b \in \alphabet} \pv[j+1]_a \pm[j+1]_{b,a} \log \frac{E_{b,a}}{\pm[j+1]_{b,a}} \\
        \text{subject to} & \pv \in \Gamma_{\alphabet}^{\mathbb{Z}_+}, \pm \in \Upsilon_{\alphabet}^{\mathbb{Z}_+}, \\
        & \pm[j]_{a,b} = 0 \text{ if } E_{a,b} = 0, \\
        & \pv[j] = \pm[j] \pv[j+1], 0 \le j < \infty,
    \end{cases}
    \tag{Problem 1}
\end{equation}
obtaining our main result.
\begin{theorem} \label{thm:pressure_monotonicity} \label{Thm: 1}
Suppose that the interaction matrix $E$ satisfies \eqref{eq:assumption}. Then, the following assertions hold.
\begin{enumerate}
    \item[(a)] $P^{(\infty)}(d,E)$ is continuous and increasing in $d$ on $(1,\infty)$.
    \item[(b)] $P^{(\infty)}(d,E)=\mathbf{P}(\aprod[d][E],E)$ for all $d \in \mathbb{N} \setminus \{1\}$.
    \item[(c)] Let $\rho(E)$ be the spectral radius of $E$ and $r_E$ be as in \eqref{5}. We have
    \begin{equation}
    \lim_{d \to 1+} P^{(\infty)}(d,E)=\log \rho(E) \text{ and } \lim_{d \to \infty} P^{(\infty)}(d,E)=\log r_E.
    \label{8}
    \end{equation}
\end{enumerate}
\end{theorem}
\noindent Heuristically speaking, the maximizing probability vectors $\pv$ and stochastic matrices $\pm$ essentially capture the characteristics of a ``typical configuration'' in $\aprod[d][E]$ as discussed in Section \ref{sec:combinatorial_optimization}. Applying this theorem, the following corollary partially generalizes a result discussed in \cite{Piantadosi2008} regarding the monotone convergence of entropy for subshifts on free groups.
\begin{corollary}
    Let $\mathcal{G}^{(d)}$ be a free group with $d$ generators, $E=E^T$ satisfy assumption \eqref{eq:assumption}, and $Y^{(d)}_E$ be a Markov shift space over $\mathcal{G}^{(d)}$. Then, the entropy $h(Y^{(d)}_E)$ is increasing in $d$.
\end{corollary}
\noindent This follows from the fact that $h(Y^{(d)}_E)=h(\aprod[(2d-1)][E])$, as is proved in \cite[Proposition 4.5]{Bana}.

A few remarks could be made at this point.
\begin{enumerate}
    \item Assumption \eqref{eq:assumption} is reasonable in that (a) by dropping all such $b \in \alphabet$ that $\sum_{a \in \alphabet} E_{a,b} = 0$, we obtain a submatrix $E'$ of $E$ with $X_{E'} = X_{E}$, leading to $\mathbf{P}(X_{E'}^{\times d},E') = \mathbf{P}(X_{E}^{\times d},E)$, and (b) similarly, by omitting such $b \in \alphabet$ that $ \sum_{b \in \alphabet} E_{a,b} = 0$, there exist a subset $\alphabet'$ of $\alphabet$ and its associated submatrix $E'$ of $E$ such that $\mathbf{P}(X_{E'}^{\times d},E')=\mathbf{P}(X_{E}^{\times d},E)$ for all $d \ge 2$. The equality follows from an essentially identical argument to that for a one-sided Markov shift.
    \item From Theorem \ref{Thm: 1}, if $E$ is a binary matrix, then $\mathbf{P}(X_{E}^{\times d},E)=h(X_{E}^{\times d})$ is increasing in $d$, as opposed to $h(\otimes_{i=1}^d X_{E})$, which is decreasing in $d$.
    \item Aside from its capability to demonstrate that $\mathbf{P}(\aprod[d][E],E)$ is increasing in $d \ge 2$, the function $d \mapsto P^{(\infty)}(d,E)$ is also interesting in that it interpolates topological pressure between tree-shifts $\aprod[d][E]$ and one-sided subshift $\aprod[][E] = \aprod[1][E]$, given that $P^{(\infty)}(1+,E)=\log \rho(E)$ coincides with the topological pressure of the potential $t \mapsto \log E_{t_1, t_0}$ defined for $t \in X_{E}$. 
\end{enumerate}

The paper is organized as follows. Section \ref{sec:preliminaries} provides background on tree shifts and introduces the key concept of pattern distribution, which is central to our exposition. Section \ref{sec:combinatorial_optimization} relates the topological pressure to the intermediate quantity $P^{(\infty)}(d,E)$ and its finite approximation $P^{(k)}(d,E)$ (maximum of \eqref{eq:limiting_optimization}), for which the discussions are summarized in Theorem \ref{thm:finite_optimizer_properties}. The proofs of Theorems \ref{thm:finite_optimizer_properties} and \ref{thm:pressure_monotonicity} are postponed until Section \ref{sec:finite_optimizer_properties} and Section \ref{sec:pressure_monotonicity}, respectively. Finally, Section \ref{sec:exp} contains two examples with their $P^{(\infty)}(d,E)$ plotted for verification purposes (Figures \ref{fig:golden-mean_top_entropy} and \ref{fig:golden-mean_top_pressure}).

\section{Preliminaries} \label{sec:preliminaries}
\subsection{Markov tree-shifts and topological pressure} \label{sec:basic_def}
To introduce the idea of pattern distribution, we should generalize topological entropy and topological pressure in a way that instead of being confined to initial $n$-subtree $\lattice[d]{n}$, we consider projections onto the set
\[
\lattice[d]{n}[m]=\cup_{i=n}^{m} \level[d]{i}.
\]
Bearing this in mind, we describe the set of \emph{blocks}, in terms of the projection map, as
\[
\block{n}[m](\aprod[d][E]) = \pi(\lattice[d]{n}[m],\aprod[d][E]):= \{(t_g)_{g \in \lattice[d]{n}[m]}: t \in \aprod[d][E]\}.
\]
The \emph{weight} of $u \in \block{n}[m](\aprod[d][E])$ on $\lattice[d]{n}[m]$ is then defined as 
\begin{equation*}
    \pblock{n}[m][u,E]=\begin{cases}
        w[u|_{\level[d]{n}}] \prod_{g \in \lattice[d]{n}[m-1]} \prod_{i=1}^{d} E_{u_{g f_i}, u_{g}} & \text{if } n < m, \\
        w[u|_{\level[d]{n}}] & \text{otherwise},
    \end{cases}
\end{equation*}
where notation $w$ is slightly abused to express the weight of a block $u \in \block{n}[n](\aprod[d][E])$:
\begin{equation*}
    w[u|_{\level[d]{n}}]=\prod_{g \in \level[d]{n}} w_{u_{g}}.
\end{equation*}
We then can introduce an analog of the partition function \eqref{eq:partition_function} defined as 
\[
\norm{\pblock{n}[m](\aprod[d][E],E)}=\sum_{u \in \block{n}[m](\aprod[d][E])} \pblock{n}[m][u,E],
\]
defined on $\lattice[d]{n}[m]$ so that $\norm{\pblock{n}(\aprod[d][E],E)}=\norm{\pblock{0}[n](\aprod[d][E],E)}$. It is noteworthy that we suppress the notation $w \in \mathbb{R}$ in all of the notations above since the limit \eqref{eq:topological_pressure} is independent of the vector. For convenience, we usually further omit the dependency on $\aprod[d][E]$, $E$, and $d$ for the above notations as long as their meanings are clear from the context.

Our conventions and notation for matrices and vectors are as follows. Recall that $\Gamma_{\alphabet}$ is the set of all probability vectors indexed by $\alphabet$ and $\Upsilon_{\alphabet}$ is the set of left stochastic matrices acting on $\Gamma_{\alphabet}$. The notation for vectors in $\Gamma_{\alphabet}$ will be in a lowercase sans serif font, such as $\pv, \tv$, while the matrices in $\Upsilon_{\alphabet}$ will be denoted in uppercase, such as $\pm, \tm$. A left stochastic matrix $\pm$ has each of its column summing to $1$, namely, $\sum_{a \in \alphabet} \pm_{a,b}=1$ for all $b \in \alphabet$. The transpose of a matrix $M$ is denoted by $M^T$, and for a vector $v$, its transpose is written as $v^T$. The product of two square matrices $M$ and $N$ (or the product of a matrix $M$ and a vector $v$) of the same dimension is denoted by $M N$ (or $M v$). For conciseness, we denote $\prod_{i=1}^k M_i = M_1 M_2 \cdots M_k$ for any $k$ square matrices of the same dimension. Finally, the standard unit vector associated with the symbol $a$ is denoted by $\stdvec{a}$, and both the all-one vector and the all-one matrix are denoted by $\onemat$. 

The spaces $\Gamma_{\alphabet}$ and $\Upsilon_{\alphabet}$ in this article are implicitly endowed with the variational distance defined as follows.
\begin{definition}
    Let $\pv, \tv \in \Gamma_{\alphabet}$ and $\pm, \tm \in \Upsilon_{\alphabet}$. The variational distance between two vectors is given by 
    \[
    \mathtt{d}_v(\pv,\tv)=\max_{S \subseteq \alphabet} |\sum_{a \in S} \pv_a-\tv_a|=\frac{1}{2} \sum_{a \in \alphabet} |\pv_a-\tv_a|,
    \]
    and the variational distance between two matrices is defined as
    \[
    \mathtt{d}_V(\pm,\tm)=\max_{b \in \alphabet} \mathtt{d}_v((\pm_{a, b})_{a \in \alphabet},(\tm_{a, b})_{a \in \alphabet}).
    \]
    Additionally, for any pairs $(\pv,\pm)$ and $(\tv,\tm)$, their variational distance is defined as
    \[
    \mathtt{d}_{v,V}((\pv,\pm),(\tv,\tm))=\max\{\mathtt{d}_v(\pv,\tv), \mathtt{d}_V(\pm,\tm)\}.
    \]
\end{definition}
\noindent For convenience, the product spaces $\Gamma_{\alphabet}^n$, $\Upsilon_{\alphabet}^n$, and $\Gamma_{\alphabet}^n \times \Upsilon_{\alphabet}^n$ are assumed throughout to be equipped with the maximum metrics, which, with a slight abuse of notation, are also denoted by $\mathtt{d}_{v}$, $\mathtt{d}_{V}$, and $\mathtt{d}_{v,V}$, respectively.

\begin{proposition} \label{prop:triangle_inequality}
    Suppose $\pv, \tv \in \Gamma_{\alphabet}$ and $\pm, \tm, \qm \in \Upsilon_{\alphabet}$. Then, the following hold. 
    \begin{itemize}
        \item $\mathtt{d}_v(\pm \pv,\tm \pv) \le \mathtt{d}_V(\pm,\tm)$
        \item $\mathtt{d}_v(\pm \pv,\pm \tv) \le \mathtt{d}_v(\pv,\tv)$
        \item $\mathtt{d}_V(\pm \qm,\tm \qm), \mathtt{d}_V(\qm \pm,\qm \tm) \le \mathtt{d}_V(\pm,\tm)$
    \end{itemize}
\end{proposition}
\begin{proof}
    It is clear that
    \begin{align*}
        \mathtt{d}_v(\pm \pv,\tm \pv)&=\frac{1}{2} \sum_{a \in \alphabet} |\sum_{b \in \alphabet} \pm_{a,b} \pv_b-\tm_{a,b} \pv_b| = \frac{1}{2} \sum_{b \in \alphabet} \pv_b \sum_{a \in \alphabet} |\pm_{a,b} -\tm_{a,b}| \le \mathtt{d}_V(\pm,\tm),
    \end{align*}
    that
    \begin{align*}
        \mathtt{d}_v(\pm \pv,\pm \tv)&=\frac{1}{2} \sum_{a \in \alphabet} |\sum_{b \in \alphabet} \pm_{a,b} \pv_b-\pm_{a,b} \tv_b| = \frac{1}{2} \sum_{b \in \alphabet} |\pv_b -\tv_b| \sum_{a \in \alphabet} \pm_{a,b} \le \mathtt{d}_v(\pv,\tv),
    \end{align*}
    and that the last inequality is a consequence of the former. 
\end{proof}
\noindent As part of our convention, every sequence $\tv, \pv, \tm, \pm$ has initial index 0 unless mentioned otherwise. Furthermore, for $\tv=(\tv[0],\tv[1],\cdots,\tv[k]) \in \Gamma_{\alphabet}^{k+1}$ and $\tm=(\tm[0],\tm[1],\cdots,\tm[k]) \in \Upsilon_{\alphabet}^{k+1}$, we write 
\[
\tv*=(\tv[k],\tv[k-1],\cdots,\tv[0]) \text{ and } \tm*=(\tm[k],\tm[k-1],\cdots,\tm[0]).
\]

\subsection{Pattern distribution} \label{sec:pattern_dist}
The primary goal of this subsection is to introduce the distribution vectors and the transition matrices, which, in essence, capture the frequency of symbol occurrence in a block, and to prepare necessary combinatorial estimates for later discussions.

For $t \in \aprod[d][E]$, the \emph{distribution vector of $t$ on level $n$} is defined as
\begin{equation}
    \dv{n}(t)=\left(\frac{\sum_{g \in \level[d]{n}} \chi_{a}(t_g)}{|\level[d]{n}|}\right)_{a \in \alphabet} \in \Gamma_{\alphabet}
\end{equation}
and $\dvset{n}(\aprod[d][E])$ denotes the set of all distribution vectors on level $n$, i.e.,
\begin{equation}
    \dvset{n}(\aprod[d][E])=\left\{\dv{n}(t): t \in \aprod[d][E]\right\}.
\end{equation}
Additionally, by writing $\childset{g}$ the set of all children of $g$, the \emph{transition matrix of $t$ from level $n$ to $n+1$} is defined as
\begin{equation} \label{eq:transition_matrix_def}
\trm{n}(t)_{a, b}=\begin{cases}
    \left(\frac{\sum_{g \in \level[d]{n}, t_g=b} \sum_{h \in \childset{g}}\chi_{a}(t_{h})}{\sum_{g \in \level[d]{n}, t_g=b} \norm{\childset{g}}}\right) & \text{if } \sum_{g \in \level[d]{n}, t_g=b} \norm{\childset{g}} > 0, \\
    \frac{E_{a, b}}{\sum_{a \in \alphabet} E_{c, b}} & \text{otherwise},
\end{cases}
\end{equation}
and $\trmset{n}(\aprod[d][E])$ stands for the set of all transition matrices:
\begin{equation}
    \trmset{n}(\aprod[d][E])=\left\{\trm{n}(t): t \in \aprod[d][E]\right\}.
\end{equation}
For the sake of convenience, we similarly generalize the notations:
\[
\dv{n}[m](t)=(\dv{n}(t),\cdots, \dv{m}(t)) \text{ and } \dvset{n}[m](\aprod[d][E])=\left\{\dv{n}[m](t): t \in \aprod[d][E]\right\},
\]
\[
\trm{n}[m](t)=(\trm{n}(t),\cdots, \trm{m-1}(t)) \text{ and } \trmset{n}[m](\aprod[d][E])=\left\{\trm{n}[m](t): t \in \aprod[d][E]\right\},
\]
and we denote by $\proddom*{n}[m](\aprod[d][E])$ the set of all compatible pairs $(\tv,\tm) \in \dvset{n}[m](\aprod[d][E]) \times \trmset{n}[m](\aprod[d][E])$, that is,
\begin{align*}
    \proddom*{n}[m](\aprod[d][E])&=\{(\dv{n}[m](t),\trm{n}[m](t)): t \in \aprod[d][E] \} \\
    &=\{(\tv,\tm) \in \dvset{n}(\aprod[d][E]) \times \trmset{n}(\aprod[d][E]): \tv[i+1]=\tm[i] \tv[i] \}.
\end{align*}
Finally, the sets of blocks with given distribution vectors and transition matrices are defined as
\begin{equation*}
    \block{n}[m](\aprod[d][E];\tv)=\{(t_g)_{g \in \lattice[d]{n}[m]}: t \in \aprod[d][E], \dv{n}[m](t)=\tv\},
\end{equation*}
\begin{equation*}
    \block{n}[m](\aprod[d][E];\tv,\tm)=\{(t_g)_{g \in \lattice[d]{n}[m]}: t \in \aprod[d][E], \dv{n}[m](t)=\tv,\trm{n}[m](t)=\tm\},
\end{equation*}
and the partition functions with given distribution vectors and transition matrices are defined as
\begin{equation*}
    \norm{\pblock{n}[m](\aprod[d][E],E;\tv)}=\sum_{u \in \block{n}[m](\aprod[d][E];\tv)} \pblock{n}[m][u,E].
\end{equation*}
\begin{equation*}
    \norm{\pblock{n}[m](\aprod[d][E],E;\tv,\tm)}=\sum_{u \in \block{n}[m](\aprod[d][E];\tv,\tm)} \pblock{n}[m][u,E].
\end{equation*}
For conciseness, we suppress the notation $\aprod[d][E]$ and $d$ whenever no ambiguity should occur. It turns out that, by a standard argument in large deviation theory, the size of these sets merely has a sub-exponential growth rate with respect to $\norm{\lattice{n}[m]}$, as is shown in the following proposition.
\begin{proposition} \label{prop:dist_set_growth_rate}
    For any $n \le m$, 
    \begin{gather*} 
        1 \le \norm{\dvset{n}[m]} \le \prod_{i=n}^{m} \left(\norm{\level{i}}+1\right)^{\norm{\alphabet}} \le \left(\frac{\norm{\lattice{n}[m]}}{m-n+1} + 1\right)^{(m-n+1) \cdot \norm{\alphabet}}, \\
        1 \le \norm{\trmset{n}[m]} \le \prod_{i=n}^{m} \left(\norm{\level{i}}+1\right)^{\norm{\alphabet} (\norm{\alphabet}+1)} \le \left(\frac{\norm{\lattice{n}[m]}}{m-n+1} + 1\right)^{2 (m-n+1) \cdot \norm{\alphabet}^2}.
    \end{gather*}
    In addition, $\lim_{\norm{\lattice{n}[m]} \to \infty} (m-n)/\norm{\lattice{n}[m]}=0$.
\end{proposition}
\begin{proof}
    The inequalities essentially follow from a simple fact: given any $k \in \mathbb{N}$, the size of the set $I_{k,\ell}:=\{(\frac{a_i}{k})_{1 \le i \le \ell}: a_i \in \mathbb{Z}_+, \sum_{i=1}^{\ell} a_i = k\}$ is no more than $(k+1)^\ell$. 
    
    For the set $\dvset{n}[m]$, the first inequality is trivial, and the second follows from that $\dvset{i} \subset I_{\norm{\level{i}},\norm{\alphabet}}$ and that $\norm{\dvset{n}[m]} \le \prod_{i=n}^{m} \norm{\dvset{i}}$. The third is a consequence of the inequality of arithmetic and geometric means.
    
    For the set $\trmset{n}[m]$, the first inequality is again trivial. For the second inequality, the set of vectors $\Lambda_i=\{(\sum_{g \in \level{i}, t_g=b} \norm{\childset{g}})_{b \in \alphabet}: t \in \aprod[d][E]\}$ is contained in the set $\norm{\level{i+1}} \cdot I_{\norm{\level{i+1}},\norm{\alphabet}}$. Consequently,
    \begin{align*}
        \norm{\trmset{i}} & \le \sum_{(\sum_{g \in \level{i}, t_g=b} \norm{\childset{g}})_{b \in \alphabet} \in \Lambda_i} \prod_{b \in \alphabet} \max \left\{\norm{I_{\sum_{g \in \level{i}, t_g=b} \norm{\childset{g}},\norm{\alphabet}}},1\right\} \\
        & \le \sum_{(\sum_{g \in \level{i}, t_g=b} \norm{\childset{g}})_{b \in \alphabet} \in \Lambda_i} \left(\norm{\level{i+1}}+1\right)^{\norm{\alphabet}^2} \\
        & \le \left(\norm{\level{i+1}}+1\right)^{\norm{\alphabet} (\norm{\alphabet}+1)}.
    \end{align*}
    The rest of the argument follows a similar line of reasoning.

    Finally, the limit of the ratio holds true as a consequence of the exponential growth of $\norm{\level{n}}$.
\end{proof}

\section{Topological pressure and pattern distribution} \label{sec:combinatorial_optimization}
This section explains the relationship between topological pressure and pattern distribution, laying the groundwork for the proof of Theorem \ref{Thm: 1}. Before diving into the exposition, we first outline the general idea. The core concept behind distribution vectors and transition matrices is simple: any two blocks $u, v \in \block{n}[m](\tv,\tm)$ have the same weight. Thus, by keeping track of $\norm{\block{n}[m](\tv,\tm)}$, we gain a clearer understanding of the partition function. The road map of our approach is as below.
\begin{description}
    \item[(Step 1)] Applying Lemma \ref{prop:dist_set_growth_rate}, we establish
    \[
    \lim_{n \to \infty} \frac{\log \norm{\pblock{n}(\aprod[d][E],E)}}{\norm{\lattice[d]{n}}} = \lim_{k \to \infty} \lim_{n \to \infty} \max_{(\tv,\tm) \in \proddom*{n}[n+k]} \frac{\log \norm{\pblock{n}[n+k](\aprod[d][E],E;\tv,\tm)}}{\norm{\lattice[d]{n}[n+k]}}.
    \]
    \item[(Step 2)] Instead of solving the maximization problem on the right-hand side, we relate it to a more tractable optimization problem \eqref{eq:limiting_optimization}, whose maximum $P^{(k)}(d,E)$ satisfies
    \[
    \lim_{k \to \infty} \lim_{n \to \infty} \max_{(\tv,\tm) \in \proddom*{n}[n+k]} \frac{\log \norm{\pblock{n}[n+k](\aprod[d][E],E;\tv,\tm)}}{\norm{\lattice[d]{n}[n+k]}} = \lim_{k \to \infty} P^{(k)}(d,E).
    \]
    \item[(Step 3)] Determine $P^{(k)}(d,E)$ for $d \in (1,\infty)$.
    \item[(Step 4)] Prove that $P^{(\infty)}(d,E) = \lim_{k \to \infty} P^{(k)}(d,E)$ and show that the function, defined for $d \in (1,\infty)$, is continuous and increasing.
\end{description}
Our argument is deeply inspired by the combinatorial proof of Cram\'er's theorem in large deviation theory (see, for example, \cite[Chapter 2]{Dembo2010a}), allowing us to apply the well-established techniques therein.

\subsection{Step 1}
This subsection studies the partition function as a function of distribution vectors and transition matrices. Noting that $\block{n}[m]$ is a disjoint union of $\block{n}[m](\tv,\tm)$ ($(\tv,\tm) \in \proddom*{n}[m]$), we may bound the partition function by
\begin{equation} \label{eq:decomp_inequality}
    \max_{\tv \in \dvset{n}[m]} \norm{\pblock{n}[m](\tv,\tm)} \le \norm{\pblock{n}[m]} \le \norm{\dvset{n}[m]} \cdot \max_{\tv \in \dvset{n}[m]} \norm{\pblock{n}[m](\tv,\tm)}.
\end{equation}
Therefore, for any unbounded increasing sequence $(\lattice{n_i}[m_i])_{i \in \mathbb{N}}$, i.e., a sequence satisfying that $\norm{\lattice{n_i}[m_i]} \le \norm{\lattice{n_{i+1}}[m_{i+1}]}$ and that $\lim_{i \to \infty} \norm{\lattice{n_i}[m_i]} = \infty$, it follows from Proposition \ref{prop:dist_set_growth_rate} that
\begin{equation} \label{eq:decomp_liminf}
    \liminf_{i \to \infty} \frac{\log \norm{\pblock{n_i}[m_i]}}{\norm{\lattice{n_i}[m_i]}} = \liminf_{i \to \infty} \max_{\tv \in \dvset{n_i}[m_i]} \frac{\log \norm{\pblock{n_i}[m_i](\tv,\tm)}}{\norm{\lattice{n_i}[m_i]}},
\end{equation}
\begin{equation} \label{eq:decomp_limsup}
    \limsup_{i \to \infty} \frac{\log \norm{\pblock{n_i}[m_i]}}{\norm{\lattice{n_i}[m_i]}} = \limsup_{i \to \infty} \max_{\tv \in \dvset{n_i}[m_i]} \frac{\log \norm{\pblock{n_i}[m_i](\tv,\tm)}}{\norm{\lattice{n_i}[m_i]}}.
\end{equation}
Following from a similar argument to \cite[Theorem 2.1]{Petersen2018b}, the upper and lower limits are known to coincide in the following two essential cases.
\begin{lemma}
    Let $E$ satisfy assumption \eqref{eq:assumption} and $(\lattice{n_i}[m_i])_{i \in \mathbb{N}}$ be an unbounded increasing sequence. Then,
    \begin{itemize}
        \item If $m_i - n_i = k$ for all $i$, then \eqref{eq:decomp_liminf} equals \eqref{eq:decomp_limsup}.
        \item If $\lim_{i \to \infty} m_i - n_i = \infty$, then both \eqref{eq:decomp_liminf} and \eqref{eq:decomp_limsup} coincide with $\mathbf{P}$.
    \end{itemize}
\end{lemma}
\begin{proof}
    For short, we denote $a_i = \log \norm{\pblock{n_i}[m_i]}/\norm{\lattice{n_i}[m_i]}$. 

    To prove the first claim, we may assume without loss of generality that $m_i = n_i + k = i + k$, since the remaining cases follow naturally from this special case since $n_i \to \infty$. Under this assumption, it is not hard to see from the definition that 
    \begin{equation} \label{eq:decreasing_sequence}
        \norm{\pblock{i+1}[i+k+1]} \le \norm{\pblock{i}[i+k]}^d \text{ and } \norm{\lattice{i+1}[i+k+1]} = d \cdot \norm{\lattice{i}[i+k]} \quad \text{for all } i. 
    \end{equation}
    Hence, $a_i$ is decreasing and, in particular, convergent.

    For the second claim, it is readily checked, following from \eqref{eq:decreasing_sequence}, that 
    \[
    \limsup_{i \to \infty} a_i \le \limsup_{i \to \infty} \frac{\log \norm{\pblock{m_i-n_i}}}{\norm{\lattice{m_i-n_i}}} = \mathbf{P},
    \]
    and thus it remains to prove $\liminf_{i \to \infty} a_i = \mathbf{P}$. To this end, we show a reversed inequality of \eqref{eq:decreasing_sequence}:
    \[
    \max_{(\tv,\tm) \in \proddom*{i}[i+k]} \norm{\pblock{i}[i+k](\tv,\tm)} \ge \max_{(\tv,\tm) \in \proddom*{0}[k]} \norm{\pblock{0}[k](\tv,\tm)}^{d^i} \quad \text{for all } i,
    \]
    immediately yielding
    \[
    \liminf_{i \to \infty} a_i \ge \liminf_{i \to \infty} \frac{\log \norm{\pblock{m_i-n_i}(\tv,\tm)}}{\norm{\lattice{m_i-n_i}}} = \mathbf{P}.
    \]
    To prove inequality, assume $(\tv,\tm) \in \proddom*{0}[k]$ is a maximizer of the right-hand side and $u^{(g)} \in \proddom*{0}[k](\tv,\tm)$, $g \in \level{i}$. Under the circumstances, $\tv[0]=\stdvec{a}$ for some $a \in \alphabet$, implying that $u^{(j)}_{\epsilon} = a$ for all $g \in \level{i}$. Due to assumption \eqref{eq:assumption}, there exist $b \in \alphabet^{i+1}$ with $E_{b_{\ell+1},b_{\ell}} > 0$ for $0 \le \ell < i$ and $b_i=a$, and thus a block $v \in \proddom*{i}[i+k](\tv,\tm)$ defined as
    \[
    v_g = \begin{cases}
        b_{\ell} & \text{if } g \in \level{\ell} \text{ for } \ell \le i, \\
        u^{(g')}_{g''} & \text{if } g = g' g'', g' \in \level{i}.
    \end{cases}
    \]
    As a result,
    \begin{equation*}
        \norm{\pblock{i:i+k}(\tv,\tm)} = \sum_{v \in \proddom*{i}[i+k](\tv,\tm)} {\pblock{}}[v,E] \ge \prod_{g \in \level{i}} \sum_{u^{(g)} \in \proddom*{0}[k](\tv,\tm)} {\pblock{}}[u^{(g)},E] = \norm{\pblock{0:k}(\tv,\tm)}^{d^{i}}
    \end{equation*}
    proving the desired result.
\end{proof}
An immediate consequence of the lemma above is
\begin{equation} \label{eq:decomp_coincidence_unbounded}
    \mathbf{P} = \lim_{k \to \infty} \lim_{n \to \infty} \max_{\tv \in \proddom*{n}[n+k]} \frac{\log \norm{\pblock{n}[n+k](\tv)}}{\norm{\lattice{n}[n+k]}},
\end{equation}
which leads to the investigation of the following optimization problem: 
\begin{equation} \label{eq:original_optimization}
    \begin{cases}
        \text{maximize} & \frac{\log \norm{\pblock{n}[n+k](\tv,\tm)}}{\norm{\lattice{n}[n+k]}} \\
        \text{subject to} & (\tv,\tm) \in \proddom*{n}[n+k]
    \end{cases}
    \tag{Problem 2}
\end{equation}
\subsection{Step 2} \label{sec:step_2}
Solving the original optimization problem \eqref{eq:original_optimization} is challenging due to its combinatorial nature. Thus, we take the following steps to transform the problem into a regular optimization problem without changing the limiting behavior of the maximum:
\begin{enumerate}
    \item Rephrase the objective function.
    \item Extend the feasible domain to a convex set of Euclidean space.
\end{enumerate}
To achieve it, we first approximate the objective function using Stirling's approximation \eqref{eq:original_optimization}. Before we demonstrate this, we first define the ``Kullback-Leibler divergence'' vector. Let $\tm \in \Upsilon_{\alphabet}$ and $M$ be a non-negative matrix of the same dimension and $\tm_{a,b}=0$ if $M_{a,b}=0$. We define such a vector as
\[
\DKL(\tm || M)_b:=\sum_{a \in \alphabet} \tm_{a, b} \log \frac{\tm_{a, b}}{M_{a, b}},
\]
where $0 \log \frac{0}{0}$ is interpreted as $0$. We should stress that $M$ is not necessarily a stochastic matrix, and so $\DKL(\tm || M)$ needs not be non-negative.
\begin{proposition}
    Let $(\tv, \tm) \in \proddom{n}[m]$. Then,
    \begin{align*}
        \frac{\log \norm{\pblock{n}[m](\tv,\tm)}}{\norm{\lattice{n}[m]}} &= \frac{\log \norm{\pblock{n}[n](\tv[0])}}{\norm{\lattice{n}[m]}} - \sum_{j=0}^{m-n-1} \frac{\norm{\level{n+j+1}}}{\norm{\lattice{n}[m_i]}} \DKL(\tm^{(j)} || E)^T \tv^{(j)} \\
        & \quad + O\left(\frac{\log \norm{\lattice{n}[m]}}{\norm{\lattice{n}[m]}}\right).
    \end{align*}
\end{proposition}
\begin{proof}
    Recall that Stirling's approximation gives
    \[
    \ln(n!) = n \log n - n + O(\ln n).
    \] 
    A combinatorial argument shows that 
    \begin{align*}
        \norm{\pblock{n}[m](\tv,\tm)}=&\norm{\pblock{n}[n](\tv[0])} \prod_{j=0}^{m-n-1} \prod_{b \in \alphabet} \left[ \binom{\norm{\level{n+j+1}} \cdot \tv[j]_b}{\norm{\level{n+j+1}} \cdot \tv[j]_b \cdot \tm[j]_{a, b}} \cdot \prod_{a \in \alphabet} E_{a,b}^{\norm{\level{n+j+1}} \cdot \tv[j]_b \cdot \tm[j]_{a, b}} \right],
    \end{align*}
    where $\binom{a}{b_1,b_2,\cdots,b_k} = \frac{a!}{b_1!b_2! \cdots b_k!}$ is the multinomial coefficient, which according to Stirling's approximation can be expressed as
    \[
    \log \binom{a}{b_1,b_2,\cdots,b_k} = - a \sum_{i=1}^k \frac{b_i}{a} \log \frac{b_i}{a} + O(\log a)
    \]
    Essentially, the multinomial coefficient in the expression of $\norm{\pblock{n}[m](\tv,\tm)}$ simply counts the blocks which have $\norm{\level{n+j+1}} \cdot \tv[j]_b$ type $b$ particles on level $n+j$ that produce $\norm{\level{n+j+1}} \cdot \tv[j]_b \cdot \tm[j]_{a, b}$ type $a$ particles on level $n+j+1$. Therefore,
    \begin{align*}
        & \frac{\log \norm{\pblock{n}[m](\tv,\tm)}}{\norm{\lattice{n}[m]}} \\
        = & \frac{\log{\norm{\pblock{n}(\tv[0])}}}{\norm{\lattice{n}[m]}}  \\
        & \hspace{1em} + \sum_{j=0}^{m-n-1} \frac{\norm{\level{n+j+1}}}{\norm{\lattice{n}[m]}} \sum_{b \in \alphabet} \left[\vphantom{\sum_{a \in \alphabet}} - \sum_{a \in \alphabet} \tv[j]_b \tm[j]_{a, b} \log\left(\frac{\tm[j]_{a, b}}{E_{a, b}}\right) + O\left(\frac{\log\norm{\level{n+j+1}}}{\norm{\level{n+j+1}}}\right)\right] \\
        = & \frac{\log{\norm{\pblock{n_i}(\tv[0])}}}{\norm{\lattice{n}[m]}} - \sum_{j=0}^{m-n-1} \frac{\norm{\level{n+j+1}}}{\norm{\lattice{n}[m]}} \DKL(\tm^{(j)}||E)^T \tv[j] + O\left(\frac{\log \norm{\lattice{n}[m]}}{\norm{\lattice{n}[m]}}\right),
    \end{align*}
    where the last equality follows from the concavity of $x \mapsto -x \log x$.
\end{proof}
\noindent As a corollary of the proposition above, we transform \eqref{eq:original_optimization} into the following.
\begin{equation} \label{eq:explicit_optimization}
    \begin{cases}
        \text{maximize} & \frac{\log \norm{\pblock{n}[n](\tv[0])}}{\norm{\lattice{n}[n+k]}} - \sum_{j=0}^{k-1} \frac{\norm{\level{n+j+1}}}{\norm{\lattice{n}[n+k]}} \DKL(\tm^{(j)}||E)^T \tv^{(j)} \\
        \text{subject to} & (\tv, \tm) \in \proddom*{n}[n+k]
    \end{cases}
    \tag{Problem 3}
\end{equation}

The trickiest part of \eqref{eq:explicit_optimization} lies in the first term involving $\norm{\pblock{n}[n](\tv[0])}$, for which, to the authors' knowledge, no good estimate is available. Nevertheless, our main interest lies in the case $k \to \infty$, where the term converges uniformly (with respect to $\tv$) to zero and hence could be safely dropped. Thus, we shall continue our discussion by rephrasing the problem, by writing $\pv=\tv*$ and $\pm=\tm*$, as
\begin{equation} \label{eq:reversed_optimization}
    \begin{cases}
        \text{maximize} & - \sum_{j=0}^{k-1} \frac{d-1}{d^{j+1}-d^{j-k}} \DKL(\pm^{(j)}||E)^T \pv^{(j+1)} \\[.7em]
        \text{subject to} & (\pv, \pm) \in \cev{\proddom*{n}[n+k]}:=\{(\tv*,\tm*): (\tv,\tm) \in \proddom*{n}[n+k]\}
    \end{cases}
    \tag{Problem 4}
\end{equation}
We remark that the bottom-up convention of \eqref{eq:reversed_optimization} turns out to be more convenient for our later discussion.

In what follows, we will further show that, as $n\to\infty$, the maximum of \eqref{eq:reversed_optimization} converges to that with the feasible domain replaced by
\begin{equation*}
    \proddom{k}=\{(\pv,\pm) \in \Gamma_{\alphabet}^{k+1} \times \Upsilon_{\alphabet}^{k}: \pm[j]_{a,b} = 0 \text{ if } E_{a,b}=0, \pv[j] = \pm[j] \pv[j+1], 0 \le j < k\},
\end{equation*}
which is a set containing $\cev{\proddom*{n}[n+k]}$ for all $n$. It is noteworthy that $\proddom{k}$ is a compact set on which the objective function is continuous, guaranteeing a maximizer for the problem defined on $\cev{\proddom*{n}[n+k]}$. Given these, our aforementioned aim could be reached if we were able to show that $\cev{\proddom*{k}[k+n]}$, as $n \to \infty$, is asymptotically dense in $\proddom{k}$. Unfortunately, it is rarely the case, since for $(\tv,\tm) \in \proddom*{k}[k+n]$, $\tv[i]_b = 0$ forces $\tm[i]_{a,b} = \frac{E_{a,b}}{\sum_{c \in \alphabet}E_{c,b}}$, whereas no such constraint is seen in $\proddom{k}$. However, it is obvious that the maximum is the same with or without the constraint, and thus we may further impose the following restrictions. 
\begin{definition}
    Let $E$ be the restriction matrix. A pair $(\tv,\tm) \in \Gamma_{\alphabet} \times \Upsilon_{\alphabet}$ is called \emph{typical} if the following hold.
    \begin{itemize}
        \item $\tm_{a,b} = 0$ if $E_{a,b} = 0$.
        \item $\tm_{a,b} = E_{a,b} \cdot (\sum_{a \in \alphabet} E_{c, b})^{-1}$ if $\tv_{b} = 0$.
    \end{itemize}
\end{definition}
\noindent Furthermore, for any fixed $\pm \in \Upsilon_{\alphabet}^{k+1}$, the objective function is just an affine function of $\pv[0] \in \Gamma_{\alphabet}$. Therefore, we can assume such maximizer, denoted by $(\pv,\pm) \in \proddom{k}$, has its last vector $\pv[k]$ taking the form of $\pv[k]=\stdvec{a}$ ($a \in \alphabet$) and lies in
\begin{multline*}
    \proddom{k}'=\{(\pv,\pm) \in (\Gamma_{\alphabet}^k \times \{\stdvec{a}\}_{a \in \alphabet}) \times \Upsilon_{\alphabet}^{k}: \pv[j] = \pm[j] \pv[j+1], \\
    (\pm[j], \pv[j+1]) \text{ is typical}, 0 \le j < k\}.
\end{multline*}
We then show in the following lemma that the set $\cev{\proddom*{n}[n+k]}$ is asymptotically dense in $\proddom{k}'$ as $n$ tends to infinity.
\begin{lemma} \label{lem:dense_limiting_domain}
    Let $(\tv',\tm) \in \dvset{n} \times \Upsilon_{\alphabet}$ be typical. Then, there exists $((\tv',\tm' \tv'), \tm') \in \dvset{n}[n+1]$ such that
    \[
    | \tm_{a,b} - \tm_{a,b}'| \le \frac{1}{\norm{\level{n+1}} \tv_{b}'} \quad \text{for all } a, b \in \alphabet.
    \]
    Moreover, $\tm'$ can be chosen so as to satisfy that ${\tm_{a,b}}'=0$ if $\tm_{a,b}=0$.
\end{lemma}
\begin{proof}
    This is a consequence of the fact that $((\tv',\tm' \tv'), \tm') \in \dvset{n}[n+1]$ if and only if $\tm' \in \Upsilon_{\alphabet}$ and that the following hold:
    \begin{itemize}
        \item $\tm_{a,b}'=\frac{c_{a,b}}{\norm{\level{n+1}} \tv_{b}'}$ for some integer $c_{a,b}$ if $\tv_{b}' \ne 0$;
        \item $\tm_{a,b}'=E_{a,b} \cdot (\sum_{a \in \alphabet} E_{c, b})^{-1}$ if $\tv_b=0$.
    \end{itemize}
    It is not hard to see that these criteria, together with the additional properties stated in the lemma, can always be satisfied simultaneously by some $\tm' \in \Upsilon_{\alphabet}$.
\end{proof}
\begin{proposition} \label{prop:dense_pattern}
    $\lim_{n \to \infty} \sup_{(\pv,\pm) \in \proddom{k}'} \mathtt{d}_{v,V}(\cev{\proddom*{n}[n+k]}, (\pv,\pm)) = 0$.
\end{proposition}
\begin{proof}
    Let $\epsilon > 0$ and $(\tv,\tm)=(\pv*,\pm*) \in \proddom{k}'$ be fixed, and denote 
    \[
    \zeta=\min\left\{\min_{i,a: \tv[i]_a > 0} \tv[i]_a, \min_{i,a,b: \tm[i]_{a,b} > 0} \tm[i]_{a,b}\right\}. 
    \]
    We take $n$ sufficiently large such that 
    \[
    \delta:=\max_{i,a: \tv[i]_a > 0} (\norm{\level{n+i+1}} \tv[i]_a)^{-1} \quad \text{ satisfying } \quad \frac{k \norm{\alphabet} \delta}{\zeta} < \min\left\{\frac{1}{2}, \epsilon\right\}.
    \]
    We then construct a pair $(\tv',\tm') \in \proddom*{n}[n+k]$ such that $\mathtt{d}_{v,V}((\cev{\tv'},\cev{\tm}'),(\tv,\tm)) < \epsilon$ through the following process. To begin with, take ${\tv[0]}' = \tv[0] \in \dvset{n}$ and choose ${\tm[0]}'$ as in Lemma \ref{lem:dense_limiting_domain} to define ${\tv[1]}' = {\tm[0]}' {\tv[0]}' \in \dvset{n+1}$. Now suppose $(({\tv[i]}')_{0 \le i \le j}, ({\tm[i]}')_{0 \le i \le j-1})$ is found and ${\tv[j]}' = 0$ if and only if $\tv[j] = 0$, then $({\tv[j]}',\tm[j]) \in \dvset{n+j} \times \Upsilon_{\alphabet}$ is typical, so we again choose ${\tm[j]}'$ as in Lemma \ref{lem:dense_limiting_domain} and set ${\tv[j+1]}' = {\tm[j]}' {\tv[j]}' \in \dvset{n+j+1}$. Otherwise, the process terminates. We claim that the process terminates only after step $k$. To see this, we make use of the following properties, which are proved in the next paragraph, for the well-defined $(({\tv[i]}')_{0 \le i \le j}, ({\tm[i]}')_{0 \le i \le j-1})$: for all $0 \le i \le j - 1$:
    \begin{enumerate}[(a)]
        \item ${\tv[i]_a}' = 0$ if and only if $\tv[i]_a = 0$, and $\mathtt{d}_v({\tv[i]}', \tv[i]) \le i \norm{\alphabet} \delta$.
        \item ${\tm[i]_{a,b}}' = 0$ if and only if $\tm[i]_{a,b} = 0$, and $\mathtt{d}_V({\tm[i]}',\tm[i]) \le \norm{\alphabet} \delta$.
    \end{enumerate}
    If the listed properties hold and if we suppose toward a contradiction that $j < k$, then we see that ${\tv[j]}'_a = ({\tm[j-1]}' {\tv[j-1]}')_a = 0$ if and only if $\tv[j]_a = (\tm[j-1] \tv[j-1])_a = 0$, and thus the process should proceed, which is a contradiction. Now that $(({\tv[j]}',{\tm[j]}'{\tv[j]}'),{\tm[j]}') \in \dvset{n+j}[n+j+1]$ for all $0 \le j < k$, which implies $(\tv,\tm) \in \proddom*{n}[n+k]$ and $\mathtt{d}_{v,V}(\cev{\proddom*{n}[n+k]}, (\pv,\pm)) < k \norm{\alphabet} \delta < \epsilon$ as desired.
    
    We prove the aforementioned claim by induction on $i$. When $i=0$, since $\tv[0]={\tv[0]}'$, property (a) is automatic. For (b), if $\tv[0]_{b} = 0$, then 
    \[
    {\tm[0]_{a,b}}' = \tm[0]_{a,b} = E_{a,b} \cdot (\sum_{a \in \alphabet} E_{c, b})^{-1}.
    \]
    If $\tv[0]_{b} > 0$, then by Lemma \ref{lem:dense_limiting_domain},
    \begin{align*}
        {\tm[0]_{a,b}}' &\ge {\tm[0]_{a,b}} - \frac{1}{\norm{\level{n+1}} \tv_{b}'} \ge \zeta - \delta > 0.
    \end{align*}
    In addition, we may apply Lemma \ref{lem:dense_limiting_domain} again to deduce $\mathtt{d}_V({\tm[0]}',\tm[0]) \le \norm{\alphabet} \delta$.

    For the induction step, we assume the hypotheses hold for $i - 1 < j - 1$. To prove (a) for the index $i$, note that ${\tv[i]}' = {\tm[i-1]}'{\tv[i-1]}' = 0$ if and only if $\tv[i-1] = \tm[i-1] \tv[i-1] = 0$ is straightforward from the induction hypotheses, and a simple calculation shows that
    \begin{multline*}
        \mathtt{d}_v({\tv[i]}', \tv[i])=\mathtt{d}_v({\tm[i-1]}'{\tv[i-1]}', {\tm[i-1]}{\tv[i-1]}) \\
        \le \mathtt{d}_V({\tm[i-1]}', {\tm[i-1]}) + \mathtt{d}_v({\tv[i-1]}', {\tv[i-1]})  \le i \norm{\alphabet} \delta.
    \end{multline*}
    For (b), we note that if ${\tv[i]_{a}}' > 0$,
    \begin{align*}
        (\norm{\level{i+1}} {\tv[i]_{a}}')^{-1} &=  \left(\norm{\level{i+1}} {\tv[i]_{a}}\left(1-\frac{{\tv[i]_a}'-\tv[i]_a)}{\tv[i]_{a}}\right)\right)^{-1} \\
        &\le \left(\norm{\level{i+1}} {\tv[i]_{a}}\left(1-\frac{\mathtt{d}_v({\tv[i]}', \tv[i])}{\tv[i]_{a}}\right)\right)^{-1} \\
        &\le \left(\norm{\level{i+1}} {\tv[i]_{a}}\left(1-\frac{i \norm{\alphabet} \delta}{\zeta})\right)\right)^{-1} \le 2 \delta.
    \end{align*}
    We apply Lemma \ref{lem:dense_limiting_domain} for the last time to obtain $\mathtt{d}_V({\tm[i]}',\tm[i]) \le \norm{\alphabet} \delta$. Furthermore, ${\tm[i]_{a,b}}' = \tm[i]_{a,b}$ if $\tv[i]_{a,b} = 0$, and if $\tv[i]_{a,b} \ne 0$, then
    \begin{align*}
        {\tm[i]_{a,b}}' &\ge {\tm[i]_{a,b}} - \mathtt{d}_V({\tm[i]}',\tm[i]) \ge \zeta - \norm{\alphabet} \delta > 0.
    \end{align*}
    Our proof by induction is therefore completed.
\end{proof}
As the last step of this subsection, we would like to once again reformulate the problem \eqref{eq:reversed_optimization} to eliminate its dependence on $k$ in the objective function. We note that the function $\DKL(\pm^{(j)}||E)^T \pv^{(j+1)}$ in \eqref{eq:reversed_optimization} is uniformly bounded for all $j$, $k$ and $\pm[j]$, $\pv[j+1]$. Furthermore, its associated coefficient $(d-1)/(d^{j+1}-d^{j-k})$ admits a limit $(d-1)/d^{j+1}$ and the ratio between the two coefficients is uniform for all $j$:
\[
\left.\frac{d-1}{d^{j+1}-d^{j-k}}\middle/\frac{d-1}{d^{j+1}}\right. = \frac{1}{1-d^{-k-1}} \to 1.
\]
As a consequence, we arrive at the following optimization problem
\begin{equation} \label{eq:limiting_optimization}
    \begin{cases}
        \text{maximize} & -\sum_{j=0}^{k-1} \frac{d-1}{d^{j+1}} \DKL(\pm[j]||E)^T \pv^{(j+1)} \\
        \text{subject to} & (\pv,\pm) \in \proddom{k}
    \end{cases}
    \tag{Problem 5}
\end{equation}
whose maximum, denoted $P^{(k)}(d,E)$, converges, due to Proposition \ref{prop:dense_pattern} and the arguments above, to the topological pressure:
\[
\mathbf{P}(\aprod[d][E],E) = \lim_{k \to \infty} P^{(k)}(d,E).
\]
It is noteworthy that the optimization problem above is closely related to \eqref{eq:infinite_optimization},
whose the maximum $P^{(\infty)}(d,E)$ is well-defined since its feasible domain is compact with respect to the metric $\mathtt{d}^{\infty}_{v,V}$ defined as
\begin{equation*}
    \mathtt{d}^{\infty}_{v,V}((\pv,\pm),(\pv',\pm'))=\sum_{j=0}^{\infty} \frac{d-1}{d^{j+1}} \mathtt{d}_{v,V}((\pv[j], \pm[j]), ({\pv[j]}', {\pm[j]}')),
\end{equation*}
and the objective function is continuous on the feasible domain. Our goal is thus to demonstrate that $P^{(k)}(d,E) \to P^{(\infty)}(d,E)$, which, as a function of $d \in (1,\infty)$, can be shown to be continuous and increasing using uniform convergence of $P^{(k)}(d,E)$ on compact subintervals. However, the proof for this requires knowledge about the solutions of \eqref{eq:limiting_optimization} and is on hold for the time being. 

\subsection{Step 3} \label{sec:step_3}
As pointed out in the last subsection, the objective function of \eqref{eq:limiting_optimization} is affine in $\pv[k]$, and thus $\pv = \stdvec{a}$ for some $a \in \alphabet$ can be assumed since $\Gamma_{\alphabet}$ is the convex hull of $\{\stdvec{a}\}_{a \in \alphabet}$. The problem simplifies to the following:
\begin{equation} \label{eq:unconstrained_optimization}
    \begin{cases}
        \text{maximize} & -\sum_{j=0}^{k-1} \frac{d-1}{d^{j+1}} \DKL(\pm[j]||E)^T \prod_{\ell=j+1}^{k-1} \pm[\ell] \pv \\
        \text{subject to} & \pv=\stdvec{a}, \pm \in \Upsilon_{\alphabet}^{k} \\
        & \pm[j]_{a,b} = 0 \text{ if } E_{a,b}=0, 0 \le j < k
    \end{cases}.
\end{equation}
We denote by $F_k(\pv,\pm;1/d,E)$ the objective function of \eqref{eq:unconstrained_optimization} and by $F_{\infty}(\pv,\pm;1/d,E)$ the objective function of \eqref{eq:infinite_optimization}, and we find a maximizer as follows. For conciseness, we suppress $E$ and $1/d$ if doing so should not occasion any confusion.
\begin{proposition} \label{prop:maximizer}
    The optimization problem \eqref{eq:unconstrained_optimization} admits an optimal transition $\pm[i]$ ($i=0,\cdots,k-1$) independent of $\pv$ that is defined as
    \[
        \pm[i]_{a,b} = \begin{cases}
            \frac{e^{\frac{d^{i+1}}{d-1}\lambda^{(i)}_a} E_{a,b}}{\sum_{c: E_{c,b}>0} e^{\frac{d^{i+1}}{d-1} \lambda^{(i)}_c} E_{c,b}} & \text{if } E_{a,b}>0, \\
            0 & \text{otherwise},
        \end{cases}
    \]
    where $\lambda^{(0)}=0$ and
    \begin{align*}
        \lambda^{(i)}_a= & \frac{d-1}{d^{i}} \log\sum_{b:E_{b,a}>0} e^{\frac{d^{i}}{d-1} \lambda^{(i-1)}_b} E_{b,a} \quad \text{ for } i=1, \cdots,k.
    \end{align*}
    Moreover, the maximum of $F_k(\pv,\pm)$ is ${\lambda^{(k)}}^T \pv$.
\end{proposition}
\begin{proof}
    The idea of proving the optimality is as follows. We construct a sequence $\pm$, independent of $\pv$ and $k$, such that given any $i=0, \cdots, k-1$ and $\tm[0],\cdots,\tm[i-1] \in \Upsilon_{\alphabet}$, $\mathsf{Z}=\pm[i]$ maximizes the function
    \[
    \mathsf{Z} \xmapsto{F_{k;\tm[i+1],\cdots,\tm[k-1]}} F_k(\pv,\pm[0],\cdots,\pm[i-1],\mathsf{Z},\tm[i+1],\cdots,\tm[k-1]).
    \]
    To find $\pm$, we should also introduce an auxiliary sequence $\lambda^{(i)}$ in our construction. By writing $\lambda^{(0)}=0 \in \mathbb{R}^{\alphabet}$, we note that $F_{k;\tm[1],\cdots,\tm[k-1]}$ is an entropy maximization problem in each column of $\mathsf{Z}$ and thus admits a maximizer (independent of $\pv$, $\tm$, and $k$) such that 
    \[
        \pm[0]_{a,b} = \begin{cases}
            \frac{e^{\frac{d}{d-1}\lambda^{(0)}_a} E_{a,b}}{\sum_{c: E_{c,b}>0} e^{\frac{d}{d-1} \lambda^{(0)}_c} E_{c,b}}=\frac{E_{a,b}}{\sum_{c \in \alphabet} E_{c,b}} & \text{if } E_{c,b} > 0, \\
            0 & \text{otherwise}.
        \end{cases}
    \]
    Now if we suppose a maximizer $\pm[i-1]$ satisfying our hypothesis is found, we let
    \begin{equation} \label{eq:lambda_iterate}
        \begin{aligned}
            \lambda^{(i)} :=& -\frac{d-1}{d^{i}} \DKL(\pm[i-1] || E)+ {\pm[i-1]}^T \lambda^{(i-1)}\\
            =& -\sum_{j=0}^{i-1} \frac{d-1}{d^{j+1}} \left(\prod_{\ell=j+1}^{i-1} \pm[\ell]\right)^T \DKL(\pm[j] || E)
        \end{aligned}
    \end{equation}
    and express $F_{k;\tm[i+1],\cdots,\tm[k-1]}$ as
    \begin{equation} \label{eq:tm_iterate}
    \begin{aligned}
        & \quad F_{k;\tm[i+1],\cdots,\tm[k-1]}(\mathsf{Z}) \\
        &= - \sum_{j=0}^{i-1} \frac{d-1}{d^{j+1}} \DKL(\pm[j] || E)^T \left(\prod_{\ell=j+1}^{i-1} \pm[\ell]\right) Z \left(\prod_{\ell=i+1}^{k-1} \tm[\ell]\right) \pv \\
        & \hspace{1em}-\frac{d-1}{d^{i+1}}\DKL(\mathsf{Z} || E)^T \prod_{\ell=i+1}^{k-1} \tm[\ell] \pv - \sum_{j=i+1}^{k-1} \frac{d-1}{d^{j+1}} \DKL(\tm[j] || E)^T \prod_{\ell=j+1}^{k-1} \tm[\ell] \pv \\
        &= \left({\lambda^{(i)}}^T \mathsf{Z} - \frac{d-1}{d^{i+1}}\DKL(\mathsf{Z} || E)^T\right) \prod_{\ell=i+1}^{k-1} \tm[\ell] \pv \\
        & \hspace{1em} - \sum_{j=i+1}^{k-1} \frac{d-1}{d^{j+1}} \DKL(\tm[j] || E)^T \prod_{\ell=j+1}^{k-1} \tm[\ell] \pv, \\
    \end{aligned}
    \end{equation}
    which also results in a classical entropy maximization problem in each column of $\mathsf{Z}$, and an optimal solution independent of $\pv$, $\tm$, and $k$ is
    \begin{equation} \label{eq:tm_expression}
        \pm[i]_{a,b} = \begin{cases}
            \frac{e^{\frac{d^{i+1}}{d-1}\lambda^{(i)}_a} E_{a,b}}{\sum_{c: E_{c,b}>0} e^{\frac{d^{i+1}}{d-1} \lambda^{(i)}_c} E_{c,b}} & \text{if } E_{a,b}>0, \\
            0 & \text{otherwise}.
        \end{cases}
    \end{equation}
    In this manner, we successfully construct the desired sequence and prove the optimality. Moreover, if we plug \eqref{eq:tm_expression} into \eqref{eq:lambda_iterate}, we deduce the recursive relation of $\lambda$, and the proof is complete.
\end{proof}
\begin{remark} \label{rmk:iteration}
    In Proposition \ref{prop:maximizer}, the vector $e^{d^{i+1}/(d-1) \lambda^{(i)}}$ satisfies $e^{d/(d-1) \lambda^{(0)}}=\onevec$ and the relation:
    \[
    e^{\frac{d^{i+1}}{d-1} \lambda^{(i)}} = (E^T e^{\frac{d^{i}}{d-1} \lambda^{(i-1)}})^{d} \text{ for } i \ge 1,
    \]
    which is essentially a generalization of the formula used in \cite[Algorithm 1]{Bana} for the computation of entropy. 
\end{remark}

\subsection{Step 4}
\subsubsection{Convergence of $P^{(k)}(d,E)$} \label{sec:finite_optimizer_properties}
We first prove the following theorem regarding the convergence of the optimization problems, preparing us for the proof of Theorem \ref{Thm: 1}. The maximizer $\pm^*$ satisfies the following properties.
\begin{theorem} \label{thm:finite_optimizer_properties}
    Let $d \in (1,\infty)$, $E$ satisfy assumption \eqref{eq:assumption},
    \[
    \alphabet_{\infty} = \{a \in \alphabet: \exists n \in \mathbb{N} \text{ such that } (E^n)_{a,a} > 0\},
    \]
    and 
    \[
    L = \min\{n \in \mathbb{N}: (E^n)_{a,a} > 0, \forall a \in \alphabet_{\infty}\}.
    \]
    Then, the maximizer found in Proposition \ref{prop:maximizer} satisfies the following properties.
    \begin{enumerate}[(a)]
        \item For every $\pv \in \Gamma_{\alphabet}$ and $k \ge 1$, $F_{k}(\pv, {\pm[0:k-1]}^*) = {\lambda^{(k)}}^T \pv$ takes values in $[(1-d^{-k})\alpha,(1-d^{-k})\beta]$, where
        \begin{align*}
            \alpha=\log \min_{b} \sum_{a} E_{a,b} \text{ and } \beta= \log \max_{b} \sum_{a} E_{a,b}.
        \end{align*}
        \item If $a, b \in \alphabet$ with $(E^i)_{a,b}>0$, then for all $j \ge 0$ and $k \ge 1$, 
        \begin{equation*}
            F_{k+i}(\stdvec{b}, {\pm[0:k+i-1]}^*) - \sum_{\ell=0}^{k+i-1}\frac{d-1}{d^{\ell+1}} \cdot \gamma \ge F_{k}(\stdvec{a}, {\pm[j:j+k-1]}^*) - \sum_{\ell=0}^{k-1}\frac{d-1}{d^{\ell+1}} \cdot \gamma,
        \end{equation*}
        where $\gamma=\log \min_{E_{a,b}>0} E_{a,b}$. In particular, for all $a \in \alphabet_{\infty}$ and all $0 \le i < L$,
        \[
        F_{L n + i}(\stdvec{a}, {\pm[0:L n + i -1]}^*) - \sum_{\ell=0}^{L n + i-1}\frac{d-1}{d^{\ell+1}} \cdot \gamma
        \] 
        is non-negative and increasing. 
        \item There exists $\pv^* \in \Gamma_{\alphabet}^{\mathbb{Z}_+}$ such that $(\pv^*,\pm^*)$ is a maximizer of \eqref{eq:infinite_optimization}. In particular,
        \begin{align*}
            &P^{(\infty)}(d,E)=F_{\infty}(\pv^*,\pm^*) \\
            =&\max_{a \in \alphabet_{\infty}}\lim_{n \to \infty} F_{k}(\stdvec{a}, {\pm[0:k-1]}^*) = \lim_{k \to \infty} \max_{a \in \alphabet_{\infty}} F_{k}(\stdvec{a}, {\pm[0:k-1]}^*) \\
            =& \lim_{k \to \infty} P^{(k)}(d,E).
        \end{align*}
    \end{enumerate}
\end{theorem}
\begin{proof}
    (a) We show by induction that the map $\lambda^{(i)} \xmapsto{f_i} \lambda^{(i+1)} $ in Proposition \ref{prop:maximizer} satisfies that 
    \[
    f_k \circ \cdots \circ f_0(\lambda^{(0)})_a \in [(1-d^{-k-1}) \alpha, (1-d^{-k-1}) \beta] \quad \text{ for all } k \ge 0. 
    \]
    which follows from the monotonicity of $f_i$: $f_i(\lambda') \ge f_i(\lambda'')$ if $\lambda' \ge \lambda''$. When $i=0$,
    \begin{align*}
        f_0(\lambda^{(0)})_a &= \frac{d-1}{d} \log \sum_{b:E_{b,a}>0} E_{b,a} \in [(1-d^{-1})\alpha,(1-d^{-1})\beta].
    \end{align*}
    Now if the hypothesis holds for $k-1$, then
    \begin{align*}
        f_k(f_{k-1} \circ \cdots \circ f_0(\lambda^{(0)}))_a & \le \frac{d-1}{d^{k+1}} \log \sum_{b:E_{b,a}>0} \beta^{(1-d^{-k}) \cdot \frac{d^{k+1}}{d-1}} E_{b,a} \le (1-d^{-k-1}) \beta,
    \end{align*}
    and a similar argument also applies to the lower bound $f_k(f_{k-1} \circ \cdots \circ f_0(\lambda^{(0)}))_a \ge (1-d^{-k-1}) \alpha$. The first item then holds by induction. 
    
    (b) For each $a, b \in \alphabet$, one can choose a sequence $(\xi_{\ell})_{\ell \ge 0}$ as follows:
    \begin{equation} \label{eq:maximizer_extension}
        (\xi_{\ell})_{\ell \ge 0}=(\xi^{a,b,i}_{\ell})_{\ell \ge 0} \text{ such that } \xi_0=a, \xi_{i}=b, E_{\xi_{\ell},\xi_{\ell+1}}>0,  \forall \ell \ge 0.
    \end{equation}  
    Then, we take transition matrices ${\pm[k+\ell]}'$ such that ${\pm[k+\ell]}'_{\xi_{\ell},\xi_{\ell+1}}=1$. Under the circumstances, we have $(\prod_{\ell=0}^{k+1} {\pm[k+j]}')_{a,b} \stdvec{b} = \stdvec{a}$. Therefore, according to \eqref{eq:unconstrained_optimization},
    \begin{align*}
        & F_{k}(\stdvec{a},{\pm[j:j+k-1]}^*) + \sum_{\ell=k}^{k+i-1}\frac{d-1}{d^{\ell+1}} \cdot \gamma \\
        = & - \sum_{\ell=0}^{k-1} \frac{d-1}{d^{\ell+1}} \DKL({\pm[j+\ell]}^*|E)^T \prod_{\ell'=j+\ell+1}^{k-1} {\pm[\ell']}^* \stdvec{a} + \sum_{\ell=k}^{k+i-1}\frac{d-1}{d^{\ell+1}} \cdot \gamma \\
        \le & F_{k+i}(\stdvec{b},({\pm[j:j+k-1]}^*,{\pm[j+k:j+k+i-1]}')) \\
        \le & F_{k+i}(\stdvec{b},{\pm[0:k+i-1]}^*).
    \end{align*}
    This finishes the proof.

    (c) We start by noting that the existence of optimizer $(\pv',\pm')$ is guaranteed by compactness of the feasible domain and continuity of the objective function. Next, by recursively replacing $\pm'$ by $\pm^*$ according to Proposition \ref{prop:maximizer}, the compactness again asserts that there exists a maximizer of the form $(\pv^*,\pm^*)$ (and thus the first equality).

    The third equality and the existence of the limits therein follow from (b). Precisely, for each $a \in \alphabet$ and each $0 \le i < L$, the following relations
    \[
    \lim_{n \ge 0} \max_{a \in \alphabet_{\infty}} \cdot = \sup_{n \ge 0} \max_{a \in \alphabet_{\infty}} \cdot = \max_{a \in \alphabet_{\infty}} \sup_{n \ge 0} \cdot = \max_{a \in \alphabet_{\infty}} \lim_{n \to \infty} \cdot
    \]
    coincide on the sequence
    \[
    F_{L n + i}(\stdvec{a},{\pm[0:Ln + i - 1]}^*) + \sum_{\ell=L n + i}^{\infty}\frac{d-1}{d^{\ell+1}} \cdot \gamma.
    \]
    To finish the proof of this part, we note the following fact due to assumption \eqref{eq:assumption}:
    \begin{equation} \label{eq:A_infinity_property}
        \text{For each } a \in \alphabet \text{ and } k \ge \norm{\alphabet}, \text{ there exists } b \in \alphabet_{\infty} \text{ such that } (E^k)_{a,b} > 0.
    \end{equation}
    Immediately this yields that
    \[
    \lim_{n \to \infty} \max_{a \in \alphabet_{\infty}} F_{L n + i}(\stdvec{a},{\pm[0:Ln + i - 1]}^*) = \lim_{n \to \infty} \max_{a \in \alphabet_{\infty}} F_{L n}(\stdvec{a},{\pm[0:Ln - 1]}^*) \text{ for all } i,
    \]
    justifying the existence of the limits and their coincidence.
    
    To prove the second equality, for each $a \in \alphabet_{\infty}$ and each $k \in \mathbb{Z}_{+}$, there exists an extension $(\pv(a,k),\pm(a,k)) \in \Gamma_{\alphabet}^{\mathbb{Z}_+} \times \Upsilon_{\alphabet}^{\mathbb{Z}_+}$ of $((\prod_{j=\ell}^{k-1} {\pm[\ell]}^* \stdvec{a})_{\ell=0}^{k},{\pm[0:k-1]}^*)$ (using the sequence $(\xi^{a,a,L}_{\ell})_{\ell}$ in \eqref{eq:maximizer_extension}) such that
    \begin{align*}
        \max_{a \in \alphabet_{\infty}} \lim_{n \to \infty} F_{k}(\stdvec{a},{\pm[0:k-1]}^*) &= \max_{a \in \alphabet_{\infty}} \lim_{n \to \infty} F_{k}(\stdvec{a},{\pm[0:k-1]}^*) + \sum_{\ell=k}^{\infty}\frac{d-1}{d^{\ell+1}} \cdot \gamma \\
        &\le \max_{a \in \alphabet_{\infty}} F_{\infty}(\pv(a,k),\pm(a,k)) \le  F_\infty(\pv^*, \pm^*),
    \end{align*}
    For the other inequality, we note that due to the affinity of $\pv \mapsto F_k(\pv,{\pm[0:k-1]}^*)$,
    \begin{equation} \label{eq:limit_existence_1}
        F_{\infty}(\pv^*,\pm^*) = \lim_{k \to \infty} F_{k}({\pv[0:k]}^*,{\pm[0:k-1]}^*) \le \liminf_{k \to \infty} \max_{a \in \alphabet} F_{k}(\stdvec{a}, {\pm[0:k-1]}^*),
    \end{equation}
    As a result, due to \eqref{eq:A_infinity_property}, we may apply (b) to deduce
    \begin{equation} \label{eq:limit_existence}
    \begin{aligned}
        & \lim_{k \to \infty} \max_{a \in \alphabet_{\infty}}F_{k}(\stdvec{a},{\pm[0:k-1]}^*) \\
        = & \lim_{k \to \infty} \max_{a \in \alphabet_{\infty}}F_{k}(\stdvec{a},{\pm[0:k-1]}^*) - \sum_{\ell=k-\norm{\alphabet}}^{k-1} \frac{d-1}{d^{\ell+1}} \cdot \gamma \\
        \ge & \limsup_{k \to \infty} \max_{a \in \alphabet} F_{k - \norm{\alphabet}}(\stdvec{a},{\pm[0:k-\norm{\alphabet}-1]}^*) \ge F_{\infty}(\pv^*,\pm^*)
    \end{aligned}
    \end{equation}
    Combining the inequalities above and the third equality proves the second equality. 
    
    Finally, we note that the last equality, together with the existence of the limit, also follows from \eqref{eq:limit_existence_1} and \eqref{eq:limit_existence}, given that $P^{(k)}(d,E) = \max_{a \in \alphabet} F_{k}(\stdvec{a},{\pm[0:k-1]}^*)$.
\end{proof}

\subsubsection{Proof of Theorem \ref{thm:pressure_monotonicity}} \label{sec:pressure_monotonicity}

\noindent Let $\beta$ and $\gamma$ be the constants defined in Theorem \ref{thm:finite_optimizer_properties}, which are independent of $d$. We show that $P^{(k)}$ converges to $P^{(\infty)}$ in a strong sense.
\begin{proof}
    (a) For our convenience, we denote $q_k(d)=d^{-k}-d^{-k-1}$ throughout the rest of the discussion. In addition, we denote by ${\pm}^*(d)$ the optimal matrices given in Proposition \ref{prop:maximizer} with the associated $\lambda^{(k)}$ denoted by ${\lambda^{(k)}}^*(d)$.
    
    Our proof of continuity and increasing property relies heavily on the uniform convergence of ${\lambda^{(Ln + i)}}^*(d)_a$: for every $a \in \alphabet_{\infty}$ and every $0 \le i \le L-1$, the function ${\Lambda^{(i)}}^*(d)_a := \lim_{n \to \infty} {\lambda^{(Ln + i)}}^*(d)_a$, which is well-defined for $d \in (1,\infty)$ according to Theorem \ref{thm:finite_optimizer_properties} (b), is continuous with the convergence uniform on every compact subinterval of $I \subset (1,\infty)$. To begin with, we note ${\lambda^{(k)}}^*(d)$ is obviously a continuously differentiable function from its definition in Proposition \ref{prop:maximizer}. To prove the uniform convergence on compact subinterval, we verify that it is an equicontinuous sequence by computing the derivative of ${\lambda^{(k)}}^*(d)$:
    \begin{equation} \label{eq:derivative_equicts}
    \begin{aligned}
        &\left({\lambda^{(k)}}^*\right)'(d) = q_k'(d) \log \left(E^T e^{q_k(d)^{-1} {\lambda^{(k-1)}}^*(d)}\right) \\
        =& q_k'(d) \log\left(E^T e^{q_k(d)^{-1} {\lambda^{(k-1)}}^*(d)}\right)+\frac{q_k(d)}{\left(E^T e^{q_k(d)^{-1} {\lambda^{(k-1)}}^*(d)}\right)} \cdot \\
        & \left[E^T \mathrm{diag}\left(-\frac{q_k'(d)}{q_k(d)^2} {\lambda^{(k-1)}}^*(d)+\frac{\left({\lambda^{(k-1)}}^*\right)'(d)}{q_k(d)} \right) e^{q_k(d)^{-1} {\lambda^{(k-1)}}^*(d)}\right] \\
        =& - q_k'(d) \DKL({\pm[k-1]}^*(d) || E) + {{\pm[k-1]}^*(d)}^T \left({\lambda^{(k-1)}}^*\right)'(d) \\
        =& \sum_{j=0}^{k-1}  (j d^{-j-1} - (j+1) d^{-j-2})  \left(\prod_{\ell=j+1}^{k-1} {\pm[\ell]}^*(d)\right)^T \DKL({\pm[j]}^*(d) | E)
    \end{aligned}
    \end{equation}
    and note that the derivative is uniformly bounded on the interval $I$:
    \[
    \left|\left({\lambda^{(Ln+i)}}^*\right)'(d)\right| \le \sup_{d \in I} \sum_{j=0}^{\infty}  |j d^{-j-1} - (j+1) d^{-j-2}| (\norm{\beta}+\norm{\gamma})< \infty.
    \]
    The equicontinuity then follows from the estimate and the mean value theorem. Now that the functions, when restricted to $I$, are clearly uniformly bounded, we apply the Arzel\`{a}-Ascoli theorem to obtain uniform convergence. 

    For the continuity of $P^{(\infty)}(d,E)$, we apply our claim to the function
    \[
    \max_{a \in \alphabet_{\infty}} {\lambda^{(Ln + i)}}^*(d)_a =  \max_{a \in \alphabet_{\infty}} F_{Ln}(\stdvec{a}, {\pm[0:Ln-1]}^*(d)). 
    \]
    to deduce uniform convergence on compact subinterval, which, together with Theorem \ref{thm:finite_optimizer_properties} (c), yields the desired result.

    To prove that $P^{(\infty)}(d,E)$ is increasing in $d$, we first associate with each $a \in \alphabet_{\infty}$ an open set $I_a$ such that 
    \[
    {\Lambda^{(Ln)}}^*(d)_a \ge {\Lambda^{(Ln)}}^*(d)_b, \text{ for all } b \in \alphabet.
    \]
    and that $\cup_a I_a$ is dense in $(1,\infty)$. This follows inductively from the fact that for any continuous functions $f_1, f_2: U \subset \mathbb{R}^n \to \mathbb{R}$ defined on some open set $U$, 
    \begin{align*}
        \{x \in U: f_1(x) > f_2(x)\} &\cup \{x \in U: f_2(x) > f_1(x)\} \\
        &\cup int(\{x \in U: f_2(x) = f_1(x)\})
    \end{align*}
    is dense in $U$. We then consider an alternative form of \eqref{eq:derivative_equicts}:
    \begin{equation} \label{eq:derivative}
        \begin{aligned}
            & \left(\stdvec{a}^T {\lambda^{(Ln)}}^*\right)'(d) \\
            =& \sum_{j=1}^{Ln-1} \sum_{i=1}^{j} (d^{-j-1}-d^{-j-2}) \stdvec{a}^T \left(\prod_{\ell=j+1}^{Ln-1} {\pm[\ell]}^*(d)\right)^T \DKL({\pm[j]}^*(d) || E) \\
            & \hspace{3em} - \sum_{j=0}^{Ln-1} d^{-j-2} \stdvec{a}^T \left(\prod_{\ell=j+1}^{Ln-1} {\pm[\ell]}^*(d)\right)^T \DKL({\pm[j]}^*(d) || E) \\
            =& -\sum_{i=1}^{Ln-1} d^{-i-1} F_{Ln-i}(\stdvec{a},{\pm[i:Ln-1]}^*(d);d,E) + \frac{d^{-2}}{1-d^{-1}} {\lambda^{(Ln)}_{a}}^*(d),
        \end{aligned}
    \end{equation}
    where the last equality follows from swapping the order of summation. If we further 
    express $d^{-2} \cdot (1-d^{-1})^{-1}=\sum_{i=1}^{\infty} d^{-i-1}$, then the above becomes
    \begin{equation*} 
        \begin{aligned}
            & \left(\stdvec{a}^T {\lambda^{(L n)}}^*\right)'(d) \\
            = & - \sum_{i=1}^{L n-1} d^{-i-1} \left[F_{Ln-i}(\stdvec{a},{\pm[i:Ln-1]}^*(d);d,E) - {\lambda^{(L n)}}^*(d)_a\right] \\
            &\hspace{3em} + \sum_{i=L n}^{\infty} d^{-i-1} {\lambda^{(L n)}}^*(d)_a, \\
            \ge & \sum_{i=1}^{L n-1} d^{-i-1} \left[{\lambda^{(L n)}}^*(d)_a - {\lambda^{(L n - i)}}^*(d)_a\right] + \sum_{i=L n}^{\infty} d^{-i-1} {\lambda^{(L n)}}^*(d)_a,
        \end{aligned}
    \end{equation*}
    where the last inequality follows again from Theorem \ref{thm:finite_optimizer_properties} (b) and the last function, due to our claim, converges uniformly on every compact subinterval to the following function:
    \[
    \sum_{i=1}^{L} \frac{d^L - 1}{d^{L+i+1}-d^{L+i}} \left(\Lambda^{(0)}(d)_a^* - \Lambda^{(L - i)}(d)_a^*\right),
    \]
    which is non-negative for all $x \in I_a$. By the mean value theorem, this provides the following bound for $[d', d''] \subset I_a$:
    \begin{align*}
        & P^{(\infty)}(d'',E)-P^{(\infty)}(d',E) \ge \limsup_{n \to \infty} \inf_{d \in [d',d'']} \left(\stdvec{a}^T {\lambda^{(L n)}}^*\right)'(d)  \cdot (d''-d') = 0.
    \end{align*}
    The proof is thus finished by noting that $\cup_{a \in \alphabet} I_a$ is open and dense in $(1,\infty)$ and $P^{(\infty)}(d,E)$ is continuous.

    (b) It follows essentially from Theorem \ref{thm:finite_optimizer_properties} (c) that $P^{(\infty)}(d,E) = \lim_{k \to \infty} P^{(k)}(d,E)$, where the limit is known to converge to $\mathbf{P}(\aprod[d][E],E)$ as discussed in Section \ref{sec:step_2}.
    
    (c) On the one hand, Theorem \ref{thm:finite_optimizer_properties} (a) guarantees $\lim_{d \to \infty} P^{(\infty)}(d,E) \le \beta$. On the other hand, Theorem \ref{thm:finite_optimizer_properties} (b) together with assumption \eqref{eq:assumption} implies, by taking $k=1$, $j=0$ and letting $i \to \infty$, that
    \begin{align*}
        \lim_{d \to \infty} p^{(\infty)}(d,E) =& \lim_{d \to \infty} \lim_{i \to \infty} \max_{b \in \alphabet} F_{i+1}(\stdvec{b}, {\pm[0:i]}^*;d,E) \\
        \ge & \lim_{d \to \infty} \lim_{i \to \infty} \max_{a \in \alphabet} F_{1}(\stdvec{a}, {\pm[0:0]}^*;d,E) + \sum_{\ell=1}^{i} (d^{-k} - d^{-k-1}) \cdot \gamma, \\
        = &\lim_{d \to \infty} (1-d^{-1}) \beta + d^{-1} \cdot \gamma = \beta
    \end{align*}
    
    For $\lim_{d \to 1+} P^{(\infty)}(d,E)$, we plug the Parry measure 
    \[
    \pm[i]_{a,b} = \frac{E_{b,a} w_{a}}{\rho(E) w_{b}} \quad \text{ and } \quad \pv_{a} = v_{a} w_{a},
    \]
    into \eqref{eq:limiting_optimization} to deduce $\lim_{d \to 1+} P^{(\infty)}(d,E) \ge \log \rho(E)$. To prove the other inequality, for every $\epsilon>0$ we first construct a new interaction matrix
    \begin{equation*}
        E'=\begin{bmatrix}
            \rho(E)+\epsilon & 0 \\
            \onevec & E
        \end{bmatrix}.
    \end{equation*}
    Since $P^{(\infty)}(d,E) \le P^{(\infty)}(d,E')$ for and $d > 1$ according to \eqref{eq:infinite_optimization}, it is sufficient to show that $\lim_{d \to 1+} P^{(\infty)}(d,E') \le \log \rho(E') = \log \rho(E)+\epsilon$. Let $v$ and $w$ be the left and right eigenvector, respectively, associated with $\rho(E')$ such that $w^T v=1$. According to the spectral decomposition of $E'$, we know
    $\lim_{n \to \infty} \rho(E')^{-n} (E')^n = v w^T$ is non-negative, and thus we can assume without loss of generality that $v=\rho(E')^{-1} E' v$ is a positive probability vector. We see from Remark \ref{rmk:iteration} that for the system defined by $E'$, there exists $C_d = \max_{a \in \alphabet'} v_a^{1-d} \ge 1$ such that 
    \begin{equation} \label{eq:uniform_bound_1-}
        \begin{aligned}
            &v^T e^{\frac{d^{k+1}}{d-1} \lambda^{(k)}} = v^T ((E')^T e^{\frac{d^{k}}{d-1} \lambda^{(k-1)}})^d \\
            \le & C_d \cdot (v^T (E')^T e^{\frac{d^{k}}{d-1} \lambda^{(k-1)}})^d = C_d \cdot \rho(E')^d (v^T e^{\frac{d^{k}}{d-1} \lambda^{(k-1)}})^d.
        \end{aligned}
    \end{equation}
    Indeed, we note that for all non-negative numbers $x_a$,
    \begin{align*}
        1 \le \frac{\sum_{a \in \alphabet'} v_a x_a^d}{(\sum_{a \in \alphabet'} v_a x_a)^d} \le \max_{a \in \alphabet'} \frac{v_a x_a^d}{v_a^d x_a^d}=\max_{a \in \alphabet'} v_a^{1-d}.
    \end{align*}
    This provides a uniform bound for $\lim_{d \to 1+} P^{(\infty)}(d,E')$, and the claim is proved by estimating the pressure using \eqref{eq:uniform_bound_1-} and letting $d \to 1+$. More specifically,
    \begin{align*}
        v^T e^{\frac{d^{k+1}}{d-1} \lambda^{(k)}} &\le C_d^{1+d+\cdots+d^{k-1}} \rho(E')^{d+d^2+\cdots+d^k} (v^T e^{\frac{d^{k}}{d-1} \lambda^{(0)}})^d \\
        &= C_d^{1+d+\cdots+d^{k-1}} \rho(E')^{d+d^2+\cdots+d^k},
    \end{align*}
    and thus $\lim_{d \to 1+} P^{(\infty)}(d,E) \le  \log (\rho(E)+\epsilon)$.
\end{proof}

\section{Experiments} \label{sec:exp}
In this section, we present two examples related to Theorem \ref{thm:pressure_monotonicity}. Consider the golden-mean tree-shifts $\aprod[d][G]$ with 
\[
G=\begin{bmatrix}
    1 & 1 \\
    1 & 0
\end{bmatrix}.
\]
We claim that the function $P^{(\infty)}(d,E)$ can be approximated by $P^{(k)}(d,E)$ with a crude estimate of error
\begin{align*}
    \frac{d^{-\ell}}{1-d^{-1}} \cdot \gamma & \le P^{(\infty)}(d,E) - P^{(k)}(d,E) \le \frac{d^{-\ell}}{1-d^{-1}} \cdot \beta.
\end{align*}
Indeed, the first inequality follows from Theorem \ref{thm:finite_optimizer_properties} (b), and the second follows, by comparing \eqref{eq:limiting_optimization} and \eqref{eq:infinite_optimization}, from the fact that each term in \eqref{eq:infinite_optimization} admits an upper bound $- \frac{d-1}{d^{j+1}} \DKL(\pm[j]|E) \pv[j+1] \le \frac{d-1}{d^{j+1}} \beta$. The figure of the topological entropy is given in Figure \ref{fig:golden-mean_top_entropy}. For the purpose of demonstration of Theorem \ref{thm:pressure_monotonicity} for general interaction matrices, we include Figure \ref{fig:golden-mean_top_pressure} to show the increasing property of pressure when 
\[
E=\begin{bmatrix}
    2 & 2 \\
    1 & 0
\end{bmatrix}.
\]
Both of the figures turn out to be consistent with Theorem \ref{thm:pressure_monotonicity} in the sense that both functions are continuous, $\lim_{d \to \infty} P^{(\infty)}(d,G)=\log_{10} 2 \approx 0.3010$, $\lim_{d \to \infty} P^{(\infty)}(d,E)=\log_{10} 3 \approx 0.4771$, $\lim_{d \to 1+} P^{(\infty)}(d,G)=\log_{10} \frac{1+\sqrt{5}}{2} \approx 0.2090$, and $\lim_{d \to 1+} P^{(\infty)}(d,E)=\log_{10} (1+\sqrt{3}) \approx 0.4365$.

\begin{figure}
    \centering
    \includegraphics{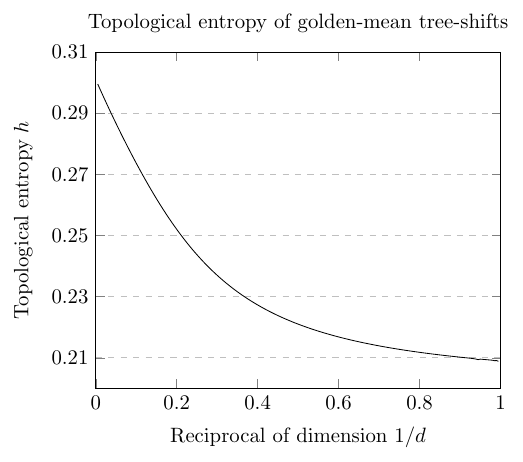}
    \caption{Topological entropy of golden-mean tree-shifts}
    \label{fig:golden-mean_top_entropy}
\end{figure}
\begin{figure}
    \centering
    \includegraphics{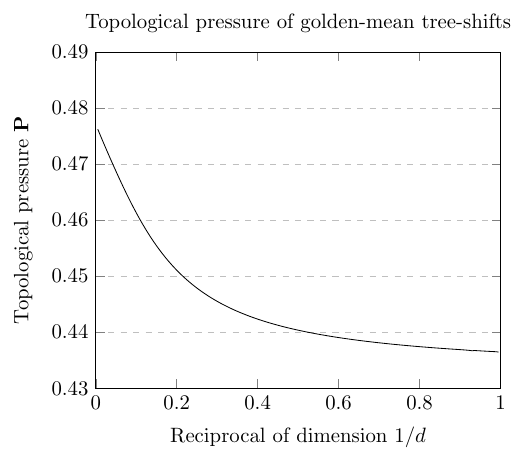}
    \caption{Topological pressure of golden-mean tree-shifts}
    \label{fig:golden-mean_top_pressure}
\end{figure}

\section*{Acknowledgments}
The first author was supported by Ministry of Science and Technology (Contract No. MOST 111-2115-M-004-005-MY3). We would like to thank the anonymous referee for the helpful comments that greatly improve the quality of the manuscript.

\bibliographystyle{amsplain_abbrv}
\bibliography{reference,ban}

\end{document}